\documentclass[11pt]{article}
\usepackage{fullpage}
\usepackage[psamsfonts]{eucal}
\usepackage{amsmath}
\usepackage{amsfonts}
\usepackage{dsfont}
\usepackage{caption}

\usepackage{amssymb}
\usepackage{amstext}
\usepackage{amscd}
\usepackage{aliascnt}
\usepackage{amsthm}
\usepackage{makeidx}
\usepackage{graphicx}
\usepackage[colorlinks=true,linkcolor=blue]{hyperref}
\usepackage{supertabular}
\usepackage{enumerate}

\usepackage{titling}
\usepackage{wasysym}
\usepackage{color}
\usepackage{xcolor}
\usepackage{comment}
\usepackage{cleveref}
\usepackage{listings}
\usepackage{relsize}

\usepackage{mathptmx}
\usepackage{microtype}

\newcommand{\Var}{\textnormal{Var}}

\usepackage{cancel}

\numberwithin{equation}{section}
 
\newcommand{\be}{\begin{equation}}
\newcommand{\ee}{\end{equation}}

\def\titlefont{\fontsize{15}{17}\bfseries\boldmath\selectfont\centering{}}
\def\authorfont{\fontsize{13}{15}}

\let\affiliationfont\rhfont

\def\address#1{\par
    {\centering{\affiliationfont#1\par}}\par\vspace*{11pt}
}

\def\body{
\setcounter{footnote}{0}
\def\thefootnote{\alph{footnote}}
\def\@makefnmark{{$^{\rm \@thefnmark}$}}
}

\def\title#1{
    \thispagestyle{plain}
    \vspace*{-14pt}
    \vskip 79pt
    {\centering{\titlefont #1\par}}
    \vskip 1em
}

\newtheorem{theorem}{Theorem}[section]

\newaliascnt{proposition}{theorem}
\newtheorem{proposition}[proposition]{Proposition}
\aliascntresetthe{proposition}

\newtheorem{thm}[theorem]{Theorem}

\newaliascnt{lemma}{theorem}
\newtheorem{lemma}[lemma]{Lemma}
\aliascntresetthe{lemma}

\newaliascnt{claim}{theorem}

\aliascntresetthe{claim}

\newaliascnt{corollary}{theorem}
\newtheorem{corollary}[corollary]{Corollary}
\aliascntresetthe{corollary}

\newaliascnt{cor}{theorem}

\aliascntresetthe{cor}
\crefname{cor}{corollary}{corollaries}
\Crefname{cor}{Corollary}{Corollaries}

\newaliascnt{conjecture}{theorem}

\aliascntresetthe{conjecture}

\newaliascnt{observation}{theorem}

\aliascntresetthe{observation}

\newaliascnt{prop}{theorem}

\aliascntresetthe{prop}
\crefname{prop}{proposition}{propositions}
\Crefname{prop}{Proposition}{Propositions}

\newtheorem*{question*}{Question}

\newaliascnt{fact}{theorem}

\aliascntresetthe{fact}

\newaliascnt{definition}{theorem}
\newtheorem{definition}[definition]{Definition}
\aliascntresetthe{definition}

\newaliascnt{dfn}{theorem}

\aliascntresetthe{dfn}
\crefname{dfn}{definition}{definitions}
\Crefname{dfn}{Definition}{Definitions}

\newaliascnt{problem}{theorem}

\aliascntresetthe{problem}

\newaliascnt{question}{theorem}

\aliascntresetthe{question}

\newtheorem*{definition*}{Definition}

\newaliascnt{example}{theorem}

\aliascntresetthe{example}

\newaliascnt{setup}{theorem}

\aliascntresetthe{setup}

\theoremstyle{remark}

\newaliascnt{remark}{theorem}
\newtheorem{remark}[remark]{Remark}
\aliascntresetthe{remark}

\newcommand{\re}{\mathrm{e}}
\newcommand{\bP}{\mathbb{P}}

\newcommand{\bE}{\mathbb{E}}
\usepackage[noadjust]{cite}

\DeclareMathOperator{\OO}{O}
\DeclareMathOperator{\oo}{o}

\newcommand{\diag}{\mathrm{diag}}

\newcommand{\bR}{\mathbb{R}}

\newcommand{\RR}{\mathbb{R}}

\newcommand{\Pois}{\operatorname{Pois}}

\newcommand{\bC}{\mathbb{C}}

\newcommand{\dist}{\mathrm{dist}}

\newcommand{\SBM}{\mathrm{SBM}}
\newcommand{\deq}{\mathrel{\mathop:}=} 

\def\@settitle{\begin{center}
    \bfseries
 \normalfont\LARGE\@title
  \end{center}
}
\def\@setauthors{\begin{center}
 \normalsize\@author
  \end{center}
}
\def\author#1{\par
    {\centering{\authorfont#1}\par\vspace*{0.05in}}
}

\newenvironment{proof*}[1][\proofname]{
  
  \begin{proof}[#1]}{\end{proof}}

\begin{document}
\title{The Bethe--Hessian down to the Percolation Threshold }
\noindent \begin{minipage}[c]{0.6\textwidth}
 \author{Dingding Dong}
\address{California Institute of Technology\\
   ddong124@caltech.edu}
 \end{minipage}
  \begin{minipage}[c]{0.2\textwidth}
 \author{Theo McKenzie}
\address{Yale University\\
   theo.mckenzie@yale.edu }
 \end{minipage}
 
\abstract{
The Bethe--Hessian is a symmetric matrix for which the negative spectrum has been observed to encode the informative structure of sparse stochastic block models. We prove that, in the stochastic block model where all vertices have expected degree $d>1$, the number of negative eigenvalues of the Bethe--Hessian is exactly the number predicted by the eigenvalues of the planted model lying outside the bulk spectrum. The condition $d>1$ is optimal, and matches a regime in which existing spectral approaches based on larger non-Hermitian matrices apply.  Our result extends a theorem of Stephan and Zhu, who established the same conclusion under the assumption $d\geq 2$.

Our improvement relies on two main ideas. First, we construct test vectors on the $2$-core, where degree fluctuations are substantially smaller, and then extend them to the entire graph while controlling the quadratic form. Second, we construct the test vectors using an isotropic basis of the underlying Markov random field, with coefficients adapted to each relevant planted eigenvalue. This allows us to control the fluctuations of the test vectors throughout the sparse regime.
 
}

\section{Introduction}

    A standard probabilistic model for community detection is the {stochastic block model} (SBM), in which vertices are assigned latent community labels and edges are generated independently with probabilities determined by those labels \cite{holland1983}. The study of the SBM has led to several influential algorithms for graph partitioning \cite{abbe2018community}. A major breakthrough in the sparse regime was the construction of efficient algorithms achieving weak recovery down to the {Kesten-Stigum threshold} \cite{massoulie2014community, mossel2018proof}, which is known to be the information-theoretic threshold for $2$ communities \cite{mossel2015reconstruction}. Around the same time, researchers observed the phenomenon of {spectral redemption}, that partitioning can be performed by examining the spectrum of the {non-backtracking operator} \cite{krzakala2013spectral}. Whereas the edge of the spectrum of sparse adjacency matrices in the SBM will localize around the highest degree vertices \cite{krivelevich2003largest}, the nonbacktracking operator remains delocalized in the sparse regime. This was proven rigorously in \cite{bordenave2018nonbacktracking}.
    
   The nonbacktracking matrix is non-Hermitian and therefore lacks the usual variational characterization and orthogonal spectral decomposition; its spectral subspaces are also more delicate to control under perturbations. Moreover, because its indices correspond to directed edges rather than vertices, its dimension can be substantially larger. These limitations motivated further study of the \emph{Bethe--Hessian}, which is the focus of this paper.

\begin{definition}[Bethe--Hessian]\label{def:Bethe--Hessian}Let $G$ be an $n$-vertex graph and $t\in\bC$. 
    The \textit{Bethe--Hessian} matrix is defined to be
\begin{align}
    H(t, G)\deq t^2I_n - tA_{ G}+ (D_{ G}-I_{n}),
\end{align}
where $A_ G$ is the adjacency matrix and $D_ G$ is the diagonal matrix of degrees in $G$. 

\end{definition}

The Bethe--Hessian has its origins in the Bethe approximation, which describes Gibbs measures in terms of local vertex and edge marginals \cite{bethe1935statistical}. The graphical-model formulation expresses this approximation through the Bethe free energy, where stationary points correspond to fixed points of belief propagation \cite{yedidia2005constructing}. For an Ising model on a graph, the Hessian of the Bethe free energy at its paramagnetic stationary point equals, up to a positive scalar factor, the Bethe--Hessian $H(t)$ \cite{saade2014spectral}. In this interpretation, negative eigenvalues of $H(t)$ give unstable directions of the paramagnetic state, providing the physical motivation for using its negative eigenspace to identify community structure. A parallel motivation comes from number theory, through the connection between the nonbacktracking operator, the Ihara zeta function, and the Ihara--Bass determinant formula \cite{hashimoto1989zeta,bass1992ihara}.

These applications led to the use of the Bethe--Hessian for graph partitioning. Numerical experiments suggest that the number of its negative eigenvalues coincides with the number of communities that can be identified above the Kesten--Stigum threshold \cite{saade2014spectral}. Because the Bethe--Hessian is a smaller Hermitian matrix, it offers a faster and more numerically stable approach to partitioning than methods based directly on the nonbacktracking matrix. Closely related work establishes the existence of informative Bethe--Hessian outliers above the Kesten--Stigum threshold in the context of robust community recovery \cite{mohanty2024robust}. Alternative choices of the Bethe--Hessian parameter, designed in particular to accommodate degree heterogeneity, have also been studied \cite{dallamico2021unified}.

Stephan and Zhu proved the prediction from \cite{saade2014spectral} rigorously for $d\geq 2$
\cite{Stephan2025community}. Their argument uses test vectors obtained
from the informative eigenvectors of the nonbacktracking operator.
When $1<d<2$ and the signal is close to the Kesten--Stigum threshold,
the degree-weighted contribution to the quadratic form for these
vectors can become too large, and the argument no longer produces a
negative direction. This can be viewed as a too high sensitivity of
the test vectors to high-degree vertices.

Near the threshold, the negative contribution associated with the planted signal becomes small. To identify it, it is necessary to control the fluctuations of the sparse local geometry sufficiently sharply to retain the sign of the quadratic form. In this paper, we completely solve the eigenvalue question by showing that the number of negative outliers of the Bethe--Hessian is exactly equal to the number of informative directions for all $d>1$, which is the optimal region (see \Cref{thm:main}).

Our argument is based on two main ideas. The first is that, for the purpose of community detection, it is enough to analyze the $2$-core of the graph. When every community has the same expected degree, the unlabeled geometry of the trees attached to the $2$-core contains no additional information about the planted labels \cite{zdeborova2016fast}. This observation has a precise counterpart for the Bethe--Hessian operator. Given a vector on the $2$-core, there is a canonical way to extend it to the full graph such that the associated quadratic form is unchanged. Moreover, for the relevant parameter $t\geq \sqrt{d}$, the squared norm of the extended vector differs from that of the original vector by an asymptotically deterministic multiplicative factor. It is therefore sufficient to determine the negative eigenspace of the Bethe--Hessian restricted to the $2$-core.

This reduction is particularly useful near the percolation threshold. As $d$ approaches $1$, the local neighborhoods of the full graph contain increasingly long and irregular pendant trees. These fluctuations make it difficult to formulate a local construction whose quadratic form remains uniformly negative near the Kesten--Stigum threshold. Passing to the $2$-core removes precisely this unstable part of the geometry. Its local neighborhoods retain a substantially more regular structure, making it possible to identify the negative directions through a local rule.

The second idea concerns the local structure of the planted labels. Following \cite{bordenave2018nonbacktracking}, one may view the labels in a locally tree-like neighborhood as a tree-indexed Markov random field, conditional on the underlying geometry. We choose a canonical isotropic decomposition of this random field and construct our test vectors using weights adapted to that decomposition. The key advantage is that, when the root is moved across an edge, the decomposition transforms in an essentially identical manner. This removes much of the dependence on the detailed geometry of the neighborhood.

As a consequence, the Bethe--Hessian quadratic form can be expressed in terms of a small collection of scalar random variables. The terms that reinforce the planted correlation can then be separated explicitly from those that oppose it. The resulting comparison reduces to a scalar inequality whose sign changes at the Kesten--Stigum threshold: when the signal is above the threshold, the correlation-enhancing terms dominate and produce a negative direction. This provides a negative subspace of the required dimension on the $2$-core, which can then be extended canonically to the full graph.

\quad

\noindent \textbf{Acknowledgements.}
T.M. was supported by a Stanford Science Fellowship. The main idea of the test vectors and analysis were developed by the authors, building upon the prior
work \cite{dong2026limit} as well as methods from \cite{bordenave2018nonbacktracking}. OpenAI's ChatGPT 5.6 Sol was used to check for errors, simplify explication, and give feedback on the manuscript.

\section{Main result}

\subsection{The Stochastic Block Model}\label{subsec:sbm}

Given a symmetric nonnegative matrix $P \in \mathbb{R}^{r \times r}$ and a distribution $\pi$ over $[r]$, the stochastic block model $\mathrm{SBM}_n(P,\pi)$ generates a random graph $G$ on $n$ vertices as follows:
\begin{enumerate}
    \item Fix a label map $\sigma:[n]\to[r]$ such that $|\sigma^{-1}(i)|=\pi_i n$ for every $i\in[r]$.
    \item For every pair of vertices $u,v\in[n]$, include the edge $uv$ independently with probability $\frac{1}{n}P_{\sigma(u),\sigma(v)}$.
\end{enumerate}

The goal is to show when it is possible to detect the original communities. Specifically, we will focus in the \emph{sparse} regime, where the expected degree of each vertex is constant. 

Throughout, we assume that all vertices in $G$ have the same expected degree $d>1$, i.e., 
\begin{align}\label{eq:const-deg}
d\deq \sum_{j=1}^r P_{ij}\pi_j>1
\qquad
\text{is the same for all } i\in[r].
\end{align}
The restriction $d>1$ is necessary, as the weak-recovery threshold is at least $1$ \cite{mossel2015reconstruction}.

\subsection{Informative eigenvalues} \label{subsec:inform-ev}
Given $\mathrm{SBM}_n(P,\pi)$, define the \textit{signal matrix} $Q\deq P\Pi$, where $\Pi\deq \diag(\pi_1,\dots,\pi_r)$. Note that $Q_{ij}=P_{ij}\pi_j$ represents the expected number of neighbors in community $j$ of a vertex in community $i$. 

As $Q$ is similar to $\Pi^{1/2}P\Pi^{1/2}$, it has real eigenvalues. We order the eigenvalues of $Q$ by absolute value as
\begin{align}\label{eq:evs}
    d=\mu_1\geq|\mu_2|\geq\cdots\geq|\mu_r|.
\end{align}

The condition $\mu_i^2>d$ is the Kesten--Stigum threshold associated
with the eigenvalue $\mu_i$. It is the natural threshold for
linear reconstruction on the local Galton--Watson tree. Indeed, at
depth $\ell$, the signal in this direction grows as $|\mu_i|^\ell$,
while the fluctuations are of order $d^{\ell/2}$. Thus the signal
dominates the noise precisely when
\[
    |\mu_i|>\sqrt d.
\]
We therefore call an eigenvalue $\mu_i$ \emph{informative} if
$\mu_i^2>d$. Denote the positive and negative informative eigenvalues of $Q$ by
\begin{align}\label{eq:pmevs}
    \mu_1^+\geq\cdots\geq\mu_{r_+}^+>\sqrt d,
    \qquad
    \mu_1^-\leq\cdots\leq\mu_{r_-}^-<-\sqrt d.
\end{align}
Thus, $r_+$ and $r_-$ are the numbers of positive and negative informative eigenvalues, respectively.

\subsection{Our result}\label{subsec:main-result}

Let $ G\sim\SBM_n(P,\pi)$, where the fixed parameters satisfy
\eqref{eq:const-deg}, and let $r_+$ and $r_-$ be defined by
\eqref{eq:pmevs}. For every $t\in\RR$,  let $N(t)$ denote the number of negative eigenvalues of $H(t, G)$; for  $\epsilon>0$,  let $N_\epsilon(t)$ denote the number of  eigenvalues of $H(t, G)$ below $-\epsilon$.

Saade, Krzakala, and Zdeborov\'a conjectured, on the basis of numerical
experiments, that the negative spectrum of the Bethe--Hessian counts
the informative directions of the SBM \cite{saade2014spectral}. Their
original formulation counts all negative eigenvalues at
$t=\pm\sqrt d$.\footnote{Our numerical experiments suggest that this
formulation requires a minor qualification: at the unshifted
parameters $t=\pm\sqrt d$, the edge of the bulk appears to contribute
an additional negative eigenvalue with probability bounded away from
zero \cite{mckenzie2026BetheHessianNumerics}. This is consistent with the behavior of other random matrix
models, in which the extremal bulk eigenvalue has a nontrivial
probability of crossing the edge of the limiting spectrum. We
therefore count only eigenvalues separated from zero.} Accordingly,
the prediction relevant to our result is
\begin{align}\label{eq:maingoal}
    N_\epsilon(\sqrt d)=r_+,
    \qquad
    N_\epsilon(-\sqrt d)=r_-.
\end{align}

Earlier results treated growing expected degree. Le and Levina
\cite{le2022estimating} proved consistency of the Bethe--Hessian
estimator in the assortative regime (i.e., the first half of \eqref{eq:maingoal}) when $d=\omega(\log n)$.
Hwang et al.\ \cite{hwang2024estimation} extended consistency to
$d=\omega_n(1)$ for an appropriate choice of the Bethe--Hessian
parameter. For constant expected degree, Stephan and Zhu \cite{Stephan2025community} subsequently
proved \eqref{eq:maingoal} for all fixed
$d\geq2$.
In this paper, we complete this program throughout the supercritical regime $d>1$. 

\begin{thm}\label{thm:main}
Assume that the SBM parameters are fixed and satisfy
\eqref{eq:const-deg}, and define $r_+$ and $r_-$ by
\eqref{eq:pmevs}. Let $ G\sim\SBM_n(P,\pi)$. There exists
$\epsilon_0>0$ such that, for every fixed
$0<\epsilon<\epsilon_0$, with probability $1-\oo_n(1)$,
\be\label{eq:fullres}
    N_\epsilon(\sqrt d)
    =
    N(\sqrt d+\epsilon)
    =
    r_+,
    \qquad
    N_\epsilon(-\sqrt d)
    =
    N(-\sqrt d-\epsilon)
    =
    r_-.
\ee
Moreover, for every $0<c<1/2$, the same equalities hold with
$\epsilon=(\log n)^{-c}$.
\end{thm}
\begin{remark}
   Stephan and Zhu proved that for $d\geq 2$, \eqref{eq:fullres} is true also for $\epsilon=(\log n)^{-c}$
for any $c>0$. While we believe \Cref{thm:main} can also be strengthened to the same range, we do not pursue it in this paper.
\end{remark}

\begin{remark}
Our theorem determines the number of negative outliers of the
Bethe--Hessian, but not the corresponding eigenspaces. In particular,
we do not prove that its eigenvectors yield weak recovery throughout
the full constant-degree regime. This was proven for sufficiently large $d$ in \cite{Stephan2025community}, and, to our knowledge, remains open for general $d>1$. 
\end{remark}

\section{Upper bounds from the nonbacktracking matrix}

We first show the upper-bound direction of \Cref{thm:main}. Namely,
with high probability,
\[
    N_\epsilon(\sqrt d),\ N(\sqrt d+\epsilon)\leq r_+,
    \qquad
    N_\epsilon(-\sqrt d),\ N(-\sqrt d-\epsilon)\leq r_-.
\]
As we will see, this follows immediately from previous work.

For a graph $G=(V,E)$, let
\[
    \vec E\deq \{(u,v):\{u,v\}\in E\}.
\]
The nonbacktracking matrix is
the $\vec E\times\vec E$ matrix
\[
    (B_G)_{(u,v),(x,y)}
    =
    \mathbf 1\{v=x,\ y\neq u\}.
\]
It is related to the Bethe--Hessian by the Ihara--Bass formula
\cite[Theorem~1.1]{kotani20002},
\begin{align}\label{eq:IharaBass}
    \det(I-zB_G)
    =z^{2|V|}(1-z^2)^{|E|-|V|}
      \det H(z^{-1},G).
\end{align}

We use the following deterministic consequence.

\begin{lemma}[{\cite[Lemma~5.4]{mohanty2024robust}}]\label{lem:ub-outlier}
Let $G$ be any graph and let $\mu>1$. The number of negative
eigenvalues of $H(\mu,G)$ is at most the number of real eigenvalues of
$B_G$ greater than $\mu$. Similarly, the number of negative
eigenvalues of $H(-\mu,G)$ is at most the number of real eigenvalues of
$B_G$ less than $-\mu$.
\end{lemma}

We also use the following description of  the outlier eigenvalues of the  SBM nonbacktracking matrix.
The location estimates follow from \cite[Theorem~4]{bordenave2018nonbacktracking},
while the fact they are real follows from \cite[Theorem~5]{Stephan2025community}.

\begin{theorem}\label{thm:nbw-evs}
Suppose that the fixed SBM parameters satisfy \eqref{eq:const-deg},
and order the eigenvalues of $Q$ as in \eqref{eq:evs}. Then, for every fixed $0<c<1$, with
probability $1-\oo_n(1)$, the matrix $B_ G$ has
$r_++r_-$ real outlier eigenvalues, such that for each eigenvalue $\mu_i$ of $Q$ with $|\mu_i|>\sqrt{d}$, there is a real eigenvalue $\lambda_i$ of $B_{ G}$ such that 
\[
    \lambda_i(B_ G)=\mu_i+\OO((\log n)^{-c}).
\]
Every remaining eigenvalue satisfies
\[
    |\lambda(B_ G)|\leq\sqrt d+\OO((\log n)^{-c}).
\]
\end{theorem}
\begin{remark}
\cite[Theorem 4]{bordenave2018nonbacktracking} states the preceding
bounds with error $\oo_n(1)$, however the proof gives the quantitative form
needed here through \cite[Proposition 8
and Proposition 11]{bordenave2018nonbacktracking}.
\end{remark}

\begin{corollary}\label{cor:thm-main-ub}
Suppose $ G\sim\SBM_n(P,\pi)$ satisfies
\eqref{eq:const-deg}--\eqref{eq:pmevs}. Then, for every sufficiently
small constant $\epsilon>0$, with high probability,
\[
    N_\epsilon(\sqrt d),\ N(\sqrt d+\epsilon)\leq r_+,
    \qquad
    N_\epsilon(-\sqrt d),\ N(-\sqrt d-\epsilon)\leq r_-.
\]
The same conclusion holds for
$\epsilon=\epsilon_n=(\log n)^{-c}$ for every $0<c<1/2$.
\end{corollary}

\begin{proof}
Let $\epsilon_*>0$ be such that for every informative eigenvalue $\mu_i$,
\[
    |\mu_i|\geq\sqrt d+\epsilon_*.
\]
Fix $0<c<1$, and suppose
\[
    (\log n)^{-c}\leq\epsilon\leq\epsilon_*/2.
\]
Applying \Cref{thm:nbw-evs} with an exponent strictly larger than $c$,
we find that, with high probability, $B_ G$ has exactly $r_+$ real
eigenvalues greater than $\sqrt d+\epsilon$ and exactly $r_-$ real
eigenvalues less than $-\sqrt d-\epsilon$. Hence
\[
    N(\sqrt d+\epsilon)\leq r_+,
    \qquad
    N(-\sqrt d-\epsilon)\leq r_-.
\]

Now fix $0<c<1/2$, let $\epsilon_n=(\log n)^{-c}$, and choose
$\gamma$ such that
\[
    c+\frac12<\gamma<1.
\]
Set $\kappa=(\log n)^{-\gamma}$. By Weyl's inequality,
\begin{align*}
    \left|
        \lambda_i(H(\sqrt d, G))
        -\lambda_i(H(\sqrt d+\kappa, G))
    \right|
    &\leq
    2\sqrt d\,\kappa+\kappa^2+\kappa\|A_ G\|.
\end{align*}
Since $\|A_ G\|=O(\sqrt{\log n})$ with high probability
\cite[Corollary~1.7(b)]{BBK19}, the right-hand side is
$o(\epsilon_n)$. Applying the previous bound  $N(\sqrt d+\kappa)\leq r_+$ gives
\[
    N_{\epsilon_n}(\sqrt d)\leq r_+.
\]
The negative case is identical. The result for fixed $\epsilon>0$
follows as $\epsilon_n<\epsilon$ for all sufficiently large $n$.
\end{proof}

\section{The $2$-core and its local limit}

We now prove the lower bound. We first treat the positive informative
eigenvalues; the negative case is reduced to this one at the end of the lower bound argument. We start with the $2$-core, where we do most of our analysis.

\begin{definition}
The $2$-core of a graph $G$, denoted by $G^{(2)}$, is its maximal
induced subgraph of minimum degree at least $2$.
\end{definition}

Let $ T_{P,\pi}$ be the rooted tree in which every vertex is assigned a type in $[r]$, constructed as follows.
The root type has distribution $\pi$, and a vertex of type $a$ has,
independently for each $b\in[r]$, $\Pois(Q_{ab})$ children
of type $b$. Given this infinite tree, we can consider its $2$-core as follows.
\begin{definition}
    Let $ T^{(2)}$ be the subtree of $ T_{P,\pi}$ induced by the
vertices lying on a bi-infinite simple path. Whenever we regard
$ T^{(2)}$ as a rooted random tree, we condition on the root belonging
to $ T^{(2)}$.
\end{definition}

As
\[
    \sum_{b=1}^r Q_{ab}=d
\]
for every type $a$, the unlabelled geometry of $ T_{P,\pi}$ is that of a
$\Pois(d)$ Galton--Watson tree. We now describe its $2$-core.

\begin{lemma}\label{lem:core-geometry}
Let $q$ be the extinction probability of a $\Pois(d)$
Galton--Watson process and set
\[
    \eta=d(1-q).
\]
Then
\begin{enumerate}
\item We have
\[
    q=\re^{-\eta}.
\]

\item
If $u$ is the parent of $v$ in $ T^{(2)}$, then, after deleting the
edge $\{u,v\}$, the component containing $v$, rooted at $v$, is a
Galton--Watson tree with offspring distribution
\be\label{eq:forward-offspring}
    \xi_1\sim\Pois(\eta)\qquad \xi_1\mid\{\xi_1\geq 1\},
\ee
with $\bE\xi_1=d$.
\item
The degree of the root in $ T^{(2)}$ has distribution
\[
    \xi_2\sim \Pois(\eta)\qquad \xi_2\mid\{\xi_2\geq2\}.
\]
\item
A $\Pois(d)$ Galton--Watson tree conditioned on extinction is
a Galton--Watson tree with offspring distribution $\Pois(dq)$.
\end{enumerate}
\end{lemma}

\begin{proof}

\begin{enumerate}
    \item 

By \cite[Proposition~5.4]{lyons2017probability}, $q$ is the smallest
solution of the fixed-point equation
\[
    q=\exp(d(q-1)).
\]
Thus
\[
    q=\exp(-d(1-q))=\re^{-\eta}.
\]
\item
By \cite[Proposition~5.28]{lyons2017probability}, if $f$ is the
probability generating function of the number of offspring, then the subtree consisting of the
vertices having an infinite line of descent, conditional on being
nonempty, has offspring generating function
\[
    \frac{f(q+(1-q)s)-q}{1-q}.
\]
As the offspring distribution is $\Pois(d)$, $f(s)=\exp(d(s-1))$, and this becomes
\[
    \frac{\re^{\eta(s-1)}-\re^{-\eta}}
         {1-\re^{-\eta}},
\]
which is the generating function of a $\Pois(\eta)$ random variable
conditioned to be positive. This proves
\eqref{eq:forward-offspring}. Furthermore,
\[
    \bE\xi
    =
    \frac{\eta}{1-\re^{-\eta}}
    =
    \frac{d(1-q)}{1-q}
    =
    d.
\]
\item 
At the original root, the number of children having an infinite line
of descent is $\Pois(\eta)$ by Poisson thinning. The root belongs to
$ T^{(2)}$ precisely when at least two such children are present,
which gives the asserted root-degree distribution.

\item
Also by \cite[Proposition~5.28]{lyons2017probability}, the process conditioned on
extinction has offspring probability generating function
\[
    \frac{f(qs)}q.
\]
In the Poisson case,
\[
    \frac{f(qs)}q
    =
    \frac{\re^{d(qs-1)}}q
    =
    \re^{dq(s-1)},
\]
which is the generating function of $\Pois(dq)$.
\end{enumerate}
\end{proof}

We next add the labels. The following lemma shows that equal-degree assumption \eqref{eq:const-deg}
 makes the survival probability of each vertex in $T$ independent of its type.

\begin{lemma}\label{lem:core-law}
The root type in $ T^{(2)}$ has distribution $\pi$. Conditional on a core vertex having type $a$, the types of its core
children are independent, each with distribution
\[
    \bP(\sigma(x)=b\mid \sigma(\text{parent of }x)=a)
    =
    \frac{Q_{ab}}d.
\]
In particular, the unlabelled geometry of $ T^{(2)}$ is independent
of its labels.

Conditional on the labelled $2$-core, the off-core trees attached to
distinct core vertices are independent. 
For every vertex $v$ in $T^{(2)}$, 
the unlabelled geometry of the descendants of $v$ in $T\setminus T^{(2)}$ is distributed as a
Galton--Watson tree with offspring distribution $\Pois(dq)$.
\end{lemma}

\begin{proof}
Suppose that a vertex has type $a$. In the whole of $T$, its numbers of
children of the different types are independent random variables with
distributions
\[
    \Pois(Q_{a1}),\ldots,\Pois(Q_{ar}).
\]
The descendant tree of each child survives with probability $1-q$,
independently of the other children. This probability does not depend
on the child's type, since every type has total offspring distribution
$\Pois(d)$.

Poisson thinning therefore shows that the numbers of 
children of types $b\in[r]$ that survive in $T^{(2)}$ are independent with distributions
\[
    \Pois\bigl((1-q)Q_{ab}\bigr).
\]
Conditional on their total number, their types are consequently
independent with probabilities
\[
    \frac{(1-q)Q_{ab}}{\sum_c(1-q)Q_{ac}}
    =
    \frac{Q_{ab}}d.
\]
The total number of surviving children has distribution
$\Pois(\eta)$, independently of the parent type. Hence conditioning on
the root belonging to the $2$-core does not change the type
distribution, which remains $\pi$. The same argument at every
generation shows that the unlabelled core geometry is independent of
the labels.

The numbers of children of type $b\in [r]$ whose descendant trees become extinct are,
by the other part of the same Poisson thinning, independent random
variables with distributions
\[
    \Pois(qQ_{ab}).
\]
They are also independent of the surviving children. The branching
property and the final assertion of \Cref{lem:core-geometry} therefore
show that, conditional on the labelled core, the deleted unlabelled
trees are independent $\Pois(dq)$ Galton--Watson trees.
\end{proof}

We now introduce some well known facts about branching processes.
\begin{lemma}\label{lem:branching-limit}
Let $Z_k$ be the generation sizes of the Galton--Watson process with
offspring distribution \eqref{eq:forward-offspring}.
Then
\begin{align}\label{eq:generation-identities}
    \bE(Z_{k-1}Z_k)&=d\,\bE Z_{k-1}^2,\nonumber\\
    \bE Z_k^2-d\,\bE(Z_{k-1}Z_k)
    &=\Var(\xi)d^{k-1}.
\end{align}

Moreover,
\[
    \frac{Z_k}{d^k}\longrightarrow W
\]
in $L^2$, where
\begin{align}\label{eq:W-moments}
    \bE W=1,
    \qquad
    \bE W^2=\frac{\eta}{d-1}<2.
\end{align}

\end{lemma}

\begin{proof}
Conditional on $Z_{k-1}$, we have
\[
    \bE[Z_k\mid Z_{k-1}]
    =
    dZ_{k-1}
\]
and
\[
    \bE[Z_k^2\mid Z_{k-1}]
    =
    d^2Z_{k-1}^2+\Var(\xi)Z_{k-1}.
\]
The first equation immediately gives $\bE[Z_k]=d^k$ and $\bE(Z_{k-1}Z_k)=d\,\bE Z_{k-1}^2$. Combining with the second, we get that
\begin{align*}
    \bE Z_k^2-d\,\bE(Z_{k-1}Z_k)=\bE[\bE[Z_k^2-d^2Z_{k-1}^2\mid Z_{k-1}]]=\bE(\Var(\xi)Z_{k-1})=\Var(\xi)d^{k-1}.
\end{align*}
This proves
\eqref{eq:generation-identities}.

For \eqref{eq:W-moments},
the $L^2$ convergence is the standard Galton--Watson martingale theorem
\cite[Theorem~8.1]{harris1963theory}. It is clear that $\bE(W)=1$. For the variance, observe that for $W_k\deq Z_k/d^k$, the above gives
\begin{align*}
    \bE(W_k^2)=\bE(W_{k-1}^2)+\frac{\Var(\xi)}{d^{k+1}}\Longrightarrow \bE(W^2)=1+\Var(\xi)\sum_{k=1}^\infty\frac{1}{d^{k+1}}=1+\frac{\Var(\xi)}{d(d-1)}.
\end{align*}
Since
\[
    \Var(\xi)=d(\eta+1)-d^2,
\]
we get that
\[
    \bE W^2
    =1+\frac{\Var(\xi)}{d(d-1)}
    =\frac{\eta}{d-1}.
\]
Using $d=\eta/(1-\re^{-\eta})$, the inequality $\bE W^2<2$ is equivalent
to
\[
    \eta-2+(\eta+2)\re^{-\eta}>0.
\]
The left-hand side vanishes at $0$ and has derivative
$1-(\eta+1)\re^{-\eta}>0$ for $\eta>0$.
\end{proof}

Another important tool is the fact that well-behaved local functions on the 2-core converge to their tree limit \cite{BJR07,Rio08}.

\begin{proposition}[ {\cite[Lemma~11.11]{BJR07}}]\label{prop:local-averages}
Fix $R\geq0$, and let $f(v,G)$ depend only on the radius-$R$
type-labelled neighborhood of $v$ in $G$, with each vertex marked
according to whether it belongs to $G^{(2)}$. Suppose that $f$ is
bounded by a fixed polynomial in the number of vertices in this
neighborhood. Then
\begin{align}\label{eq:local-average}
    \frac1{|V( G^{(2)})|}
    \sum_{v\in V( G^{(2)})}f(v, G)
    \overset \bP\longrightarrow
    \bE\left[
        f(o, T_{P,\pi})
        \,\middle|\,
        o\in V( T^{(2)})
    \right],
\end{align}
where the vertices of $ T_{P,\pi}$ are marked according to membership
in its $2$-core.
\end{proposition}

\section{Test vectors and the limiting quadratic form}

We first set up the Markov random field. As $Q=P\Pi$, we can choose eigenvectors $\phi_1,\ldots,\phi_{r_+}$ of $Q$ corresponding to
$\mu_1^+,\ldots,\mu_{r_+}^+$ so that
\begin{align}\label{eq:pi-orthonormal}
    \sum_{a=1}^r\pi_a\phi_i(a)\phi_j(a)=\delta_{ij}.
\end{align}
Fix any $i\in[r_+]$, and let  $\Gamma$ be a graph with labeling function $\sigma:V(\Gamma)\rightarrow [r]$. For vertex $u\in V(\Gamma)$, we set
\[
    \rho_i\deq\frac{\mu_i^+}{d},
    \qquad
    X_{u,i}\deq\phi_i(\sigma(u)).
\]

Further, for $u\in V(\Gamma)$, let
\[
    Z_k(u)\deq|\{y:\dist(u,y)=k\}|,
\]
and, for vertices $u,v\in V(\Gamma)$, let
\[
    Z_k(u\mid v)
    \deq |\{y:\dist(u,y)=k,\ \dist(v,y)=\dist(v,u)+k\}|.
\]

We can now define the negative direction of the Bethe--Hessian. Fix $\sqrt d<\mu_0<d$. For $L\geq0$, a graph $\Gamma$, and a vertex $o$ such that the radius-$L$ neighborhood
$B_L(o)$ is a tree, define the vector $\psi_{i,L}\in \bR^{|V(\Gamma)|}$ such that for $o\in V(\Gamma)$, 
\begin{align}\label{eq:test-vector}
\begin{split}
    \psi_{i,L}(o)
    &\deq
    Z_L(o)X_{o,i}\\
    &\quad+
    \sum_{\ell=1}^L
    \left(\frac d{\mu_0}\right)^{\ell-1}
    \sum_{\dist(o,x)=\ell}
    Z_{L-\ell}(x\mid o)
    \left(
        \rho_i^{-1}X_{x,i}-X_{p_o(x),i}
    \right),
\end{split}
\end{align}
where  $p_o(x)$ denotes the predecessor of $x$ on the shortest path from $o$ to $x$. 
We set $\psi_{i,L}(o)\deq0$ whenever $B_L(o)$ is not
a tree. 

The increment    $\rho_i^{-1}X_{x,i}-X_{p_o(x),i}$
is orthogonal to the labels revealed before reaching $x$. The
coefficients in \eqref{eq:test-vector} are chosen so that the
contributions from consecutive levels will cancel in the expected
Bethe--Hessian quadratic form.

The first graph we consider is $T^{(2)}$, which inherits the labeling function from $T$. The next result gives the covariance of the Markov random fields on $T^{(2)}$.

\begin{lemma}\label{lem:innovation-orthogonality}
Conditional on the unlabelled tree, for fixed $i$,
\[
    \bE[X_{u,i}X_{v,i}\mid T^{(2)}]=\rho_i^{\dist(u,v)},
    \qquad
    \ \bE[X_{u,i}X_{v,j}\mid T^{(2)}]=0\quad(i\neq j),
\]
and the variables $X_{o,i}$ and
\[
    \rho_i^{-1}X_{x,i}-X_{p_o(x),i},
    \qquad x\neq o
\]
are pairwise orthogonal in expectation. 
\end{lemma}

\begin{proof}
By \cite[Lemma 27]{bordenave2018nonbacktracking}, we know that $\bE[X_{u,i}X_{v,j}\mid T^{(2)}]=0$ for all $i\neq j$. Also  for all $u\sim v$ in $T^{(2)}$, we have 
\[
\bP(\sigma(u)=b\mid \sigma(v)=a,\, T^{(2)})=\frac{Q_{ab}}{d}.
\]
Thus, for all $u,v$ with $\dist(u,v)=\ell$, we have 
    \begin{align*}
        \bE[X_{u,i}X_{v,i}\mid  T^{(2)}]&=\sum_{a_0,a_\ell\in[r]}\phi_i(a_0)\phi_i(a_\ell)\sum_{a_1,\dots,a_{\ell-1}}\pi_{a_0}\left(\frac{Q_{a_0a_1}}{d}\right)\cdots \left(\frac{Q_{a_{\ell-1}a_\ell}}{d}\right)\\
        &=d^{-\ell}\sum_{a_0\in[r]}\pi_{a_0}\phi_i(a_0)\sum_{a_\ell\in[r]} (Q^\ell)_{a_0a_\ell}\phi_i(a_\ell)=d^{-\ell}\sum_{a_0\in[r]}\pi_{a_0}\phi_i(a_0) (Q^\ell \phi_i)(a_0)\\
        &=d^{-\ell}\sum_{a_0\in[r]}\pi_{a_0}\phi_i(a_0)\cdot  (\mu_i^+)^\ell \phi_i(a_0)=\left(\frac{\mu_i^+}{d}\right)^\ell\sum_{a_0\in[r]}\pi_{a_0}\phi_i(a_0)^2=\left(\frac{\mu_i^+}{d}\right)^\ell=\rho_i^\ell.
    \end{align*}

Now fix $x\neq x'$ such that none of $x,x'$ is $o$. The above gives
\begin{align*}
    &\bE ((\rho_i^{-1}X_{x,i}-X_{p_o(x),i})(\rho_i^{-1}X_{x',i}-X_{p_o(x'),i})\mid  T^{(2)})\\
    &=\rho_i^{\dist(x,x')-2}-\rho_i^{\dist(p_o(x),x')-1}-\rho_i^{\dist(x,p_o(x'))-1}+\rho_i^{\dist(p_o(x),p_o(x'))}.
\end{align*}
It is easy to check that, regardless of whether $o$ lies between $x$ and $x'$, the above is always 0. Thus, the variables  $(\rho_i^{-1}X_{x,i}-X_{p_o(x),i})_{x\neq 0}$ are pairwise orthogonal in expectation. Extending the same argument to $X_{o,i}$ shows that all variables in $\{X_{o,i}\}\cup\{\rho_i^{-1}X_{x,i}-X_{p_o(x),i}:x\neq o\}$ are pairwise orthogonal in expectation.
 \end{proof}

To show that the test vectors give negative directions, we decompose
their quadratic forms into local contributions. For $t>0$, define
\begin{align}\label{eq:local-U}
\begin{split}
    U_{i,t,L}(o)
    &\deq 
    -\frac1t\psi_{i,L}(o)[H(t)\psi_{i,L}](o)\\
    &=
    \psi_{i,L}(o)\sum_{x\sim o}\psi_{i,L}(x)
    -
    \left(\frac{d_o-1}{t}+t\right)\psi_{i,L}(o)^2.
\end{split}
\end{align}
This definition is chosen so that, on every finite graph,
\begin{align}\label{eq:local-to-global-U}
    -\frac1t
    \psi_{i,L}^{\top}H(t)\psi_{i,L}
    =
    \sum_o U_{i,t,L}(o).
\end{align}
Thus a positive lower bound on $\bE U_{i,t,L}(o)$ will yield a
negative direction for the Bethe--Hessian.

\begin{proposition}\label{prop:tree-negative}
There are constants $\mu_0>\sqrt d$, $\epsilon_0>0$, $c>0$, and an
integer $L$, depending only on the fixed SBM parameters, such that,
for every $i\in[r_+]$ and every
\[
    \sqrt d\leq t\leq\sqrt d+\epsilon_0,
\]
we have that on $T^{(2)}$,
\begin{align}\label{eq:tree-negative}
    \bE U_{i,t,L}(o)\geq c.
\end{align}
\end{proposition}

\begin{proof}
Fix $i$, and write $\rho=\rho_i$ and
$\psi_L=\psi_{i,L}$.

The proof has four steps. We first expand the norm and adjacency terms in the orthogonal innovation basis. We then cancel the interior generations against the diagonal Bethe--Hessian terms. Then, we show the remaining bulk discrepancy is nonnegative except for a controllable first-level contribution. Finally, we show that the contribution of an edge is strictly positive precisely above the Kesten–Stigum threshold.

We first compute the two conditional expectations appearing in
\eqref{eq:local-U}. By
\Cref{lem:innovation-orthogonality}, the summands in
\eqref{eq:test-vector} are pairwise orthogonal in expectation conditional on  the
unlabelled tree. Therefore,
\begin{align}\label{eq:test-vector-norm}
\begin{split}
    \bE\left[
        \psi_L(o)^2
        \mathrel{\big|}T^{(2)}
    \right]
    &=
    Z_L(o)^2\\
    &\quad+
    (\rho^{-2}-1)
    \sum_{\ell=1}^L
    \left(\frac d{\mu_0}\right)^{2\ell-2}
    \sum_{\dist(o,x)=\ell}
        Z_{L-\ell}(x\mid o)^2.
\end{split}
\end{align}
Here, the term $\rho^{-2}-1$ comes from the fact that 
\[
\bE ((\rho^{-1}X_{x,i}-X_{p_o(x),i})^2\mid T^{(2)})=\rho^{-2}-2+1=\rho^{-2}-1.
\]
Since $L$ is fixed, the right-hand side of \eqref{eq:test-vector-norm} has finite expectation. It is
also strictly positive, since $Z_L(o)>0$ on $ T^{(2)}$.

For the adjacency term, we once again track coefficients in the orthogonal basis $(\rho_i^{-1}X_{x,i}-X_{p_o(x),i})$. Again applying \Cref{lem:innovation-orthogonality}, surviving terms can be grouped as
\begin{align}\label{eq:conditional-adjacency}
\begin{split}
    &\bE\left[
        \psi_L(o)\sum_{x\sim o}\psi_L(x)
        \mathrel{\big|}T^{(2)}
    \right]\\
    &=
    \rho\sum_{x\sim o}Z_L(o)Z_L(x)\\
    &\quad+
    (\rho^{-1}-\rho)
    \sum_{x\sim o}
    \Bigl(
        Z_L(o)Z_{L-1}(o\mid x)
        +
        Z_{L-1}(x\mid o)Z_L(x)\\
    &\hspace{47mm}
        -
        Z_{L-1}(x\mid o)Z_{L-1}(o\mid x)
    \Bigr)\\
    &\quad+
    (\rho^{-2}-1)
    \sum_{\ell=2}^{L}
    \left(\frac d{\mu_0}\right)^{2\ell-3}
    \sum_{\dist(o,x)=\ell}
        Z_{L-\ell+1}(x\mid o)
        Z_{L-\ell}(x\mid o)\\
    &\quad+
    (\rho^{-2}-1)
    \sum_{\ell=1}^{L-1}
    \left(\frac d{\mu_0}\right)^{2\ell-1}
    \sum_{\dist(o,x)=\ell}
        (d_o-1)
        Z_{L-\ell-1}(x\mid o)
        Z_{L-\ell}(x\mid o).
\end{split}
\end{align}
For each neighbor
$x$ of $o$, collect the terms involving the two branches obtained by
deleting the edge $\{o,x\}$ into
\begin{align}\label{eq:edge-contribution}
\begin{split}
    g_{t,L}(o,x)
    &\deq 
    \left(
        \rho Z_L(x)
        +
        (\rho^{-1}-\rho)Z_{L-1}(o\mid x)
        -
        tZ_{L-1}(x\mid o)\right.\\
    &\hspace{34mm}\left.
        -
        \frac{Z_L(o)-Z_{L-1}(x\mid o)}t
    \right)Z_L(o)\\
    &\quad+
    (\rho^{-1}-\rho)
    \left(
        Z_L(x)
        -
        Z_{L-1}(o\mid x)
        -
        \frac t\rho Z_{L-1}(x\mid o)
    \right)
    Z_{L-1}(x\mid o).
\end{split}
\end{align}
We will see that $\sum_{x\sim o}g_{t,L}(o,x)$ covers most of the  terms  in $\bE U_{i,t,L}(o)$ that involve $o$ and $x\sim o$. Before that, we make some preparation to 
control the remaining terms.

 For $1\leq\ell\leq L$, write
\[
    S_\ell(o)
    \deq 
    (\rho^{-2}-1)
    \left(\frac d{\mu_0}\right)^{2\ell-2}
    \sum_{\dist(o,x)=\ell}
        Z_{L-\ell}(x\mid o)^2.
\]
Then \eqref{eq:test-vector-norm} becomes
\[
    \bE\left[
        \psi_L(o)^2
        \mathrel{\big|}T^{(2)}
    \right]
    =
    Z_L(o)^2+\sum_{\ell=1}^L S_\ell(o).
\]

We first compare the two interior sums in
\eqref{eq:conditional-adjacency} with the corresponding terms in this
conditional second moment. Fix $2\leq\ell\leq L$. Conditional on the tree through level $\ell$, the
processes growing from the vertices at that level are independent
copies of the Galton--Watson process in
\Cref{lem:branching-limit}. Since
\[
    \bE(Z_{L-\ell+1}Z_{L-\ell})=d\,\bE Z_{L-\ell}^2,
\]
we obtain
\begin{align}\label{eq:first-interior-cancellation}
\begin{split}
    &(\rho^{-2}-1)
    \left(\frac d{\mu_0}\right)^{2\ell-3}
    \bE\sum_{\dist(o,x)=\ell}
        Z_{L-\ell+1}(x\mid o)
        Z_{L-\ell}(x\mid o)
    =
    \mu_0\,\bE S_\ell(o).
\end{split}
\end{align}

Similarly, for $1\leq\ell\leq L-1$, the second identity in
\eqref{eq:generation-identities} gives
\begin{align}\label{eq:second-interior-cancellation}
\begin{split}
    &\frac1{\mu_0}
    \bE\left[(d_o-1)S_\ell(o)\right]\\
    &\quad-
    (\rho^{-2}-1)
    \left(\frac d{\mu_0}\right)^{2\ell-1}
    \bE\sum_{\dist(o,x)=\ell}
        (d_o-1)
        Z_{L-\ell-1}(x\mid o)
        Z_{L-\ell}(x\mid o)\\
    &\qquad
    =
    \frac{\rho^{-2}-1}{\mu_0}
    \left(\frac d{\mu_0}\right)^{2\ell-2}
    \Var(\xi)d^{L-\ell-1}
    \bE\left[(d_o-1)Z_\ell(o)\right].
\end{split}
\end{align}
The right-hand side is nonnegative.

We will also use the following immediate consequence of the
root-degree law. If $d_o$ is the degree of the root in
$ T^{(2)}$, then we claim
\begin{align}\label{eq:core-degree-identity}
\bE[d_o(d_o-1)]
=
d\bE d_o.
\end{align}
Indeed, if $D\sim\Pois(\eta)$, then
\[
\bE[D(D-1)\mid D\geq2]
=
\frac{\eta^2}{\bP(D\geq2)}
\]
and
\[
\bE[D\mid D\geq2]
=
\frac{\eta(1-\re^{-\eta})}{\bP(D\geq2)}.
\]
The ratio of these two quantities is
\[
\frac{\eta}{1-\re^{-\eta}}
=
d.
\]

Conditional on $d_o$,
the $\ell$th generation is the union of $d_o$ independent forward
branches. Hence, by \eqref{eq:core-degree-identity},
\[
\begin{aligned}
    \bE\left[(d_o-1)Z_\ell(o)\right]
    &=
    \bE[d_o(d_o-1)]d^{\ell-1}=
    d^\ell\bE d_o.
\end{aligned}
\]
It follows that the sum over $\ell$ of the right-hand side of
\eqref{eq:second-interior-cancellation} is at most
\begin{align}\label{eq:interior-error}
    C d^{2L}
    \left(\frac d{\mu_0^2}\right)^L.
\end{align}

We now substitute \eqref{eq:test-vector-norm} and
\eqref{eq:conditional-adjacency} into \eqref{eq:local-U}. Since
\[
    \sum_{x\sim o}Z_{L-1}(x\mid o)=Z_L(o),
\]
we can sum $g_{t,L}(o,x)$ as defined in \eqref{eq:edge-contribution} over $x\sim o$. Thus, we obtain that the sum first two terms of \eqref{eq:conditional-adjacency} corresponding to the nearest neighbors is equal to
\[
\sum_{x\sim o}g_{t,L}(o,x)+\left(\left(\frac{d_o-1}{t}+t\right)Z_L(o)^2+tS_1(o)\right).
\]
Using \eqref{eq:first-interior-cancellation},
\eqref{eq:second-interior-cancellation}, and
\eqref{eq:interior-error}, we therefore obtain
\begin{align}\label{eq:quadratic-intermediate}
\begin{split}
    \bE U_{i,t,L}(o)
    &\geq
    \bE\sum_{x\sim o}g_{t,L}(o,x)\\
    &\quad+
    (\mu_0-t)
    \sum_{\ell=2}^{L-1}
    \bE\left[
        \left(
            1-\frac{d_o-1}{t\mu_0}
        \right)S_\ell(o)
    \right]\\
    &\quad-
    \frac{\mu_0-t}{t\mu_0}
    \bE\left[(d_o-1)S_1(o)\right]\\
    &\quad-
    C d^{2L}
    \left(\frac d{\mu_0^2}\right)^L.
\end{split}
\end{align}
Here the terms with $\ell=L$ have also been absorbed into the
last error. Indeed, since
\[
    S_L(o)
    =
    (\rho^{-2}-1)
    \left(\frac d{\mu_0}\right)^{2L-2}
    Z_L(o),
\]
we  know from the root-degree moment identities \eqref{eq:core-degree-identity} that 
\[
 \bE S_L(o)+   \bE\left[(d_o-1)S_L(o)\right]
    \leq
    C d^{2L}
    \left(\frac d{\mu_0^2}\right)^L.
\]

We next show that the second line of
\eqref{eq:quadratic-intermediate} is nonnegative. Conditional on
$d_o$, each $S_\ell(o)$ is the sum of $d_o$ independent nonnegative
contributions, one from each branch at the root. If $m_\ell$ is the
expected contribution of one branch, then
\begin{align*}
    &\bE\left[
        \left(
            1-\frac{d_o-1}{t\mu_0}
        \right)S_\ell(o)
    \right]\\
    &\qquad
    =
    \left(
        \bE d_o
        -
        \frac{\bE[d_o(d_o-1)]}{t\mu_0}
    \right)m_\ell\\
    &\qquad
    =
    \bE d_o
    \left(
        1-\frac d{t\mu_0}
    \right)m_\ell
    \geq0.
\end{align*}
Here we used \eqref{eq:core-degree-identity} and
$t\mu_0>d$. We will also choose $t<\mu_0$ so the factor of 
$\mu_0-t$ in \eqref{eq:quadratic-intermediate} is also nonnegative.

Finally,
\[
    S_1(o)
    =
    (\rho^{-2}-1)
    \sum_{x\sim o}Z_{L-1}(x\mid o)^2.
\]
Conditional on $d_o$, the variables $Z_{L-1}(x\mid o)$ are
independent copies of the $(L-1)$st generation size. Therefore,
\begin{align*}
    \bE\left[(d_o-1)S_1(o)\right]
    &=
    (\rho^{-2}-1)
    \bE[d_o(d_o-1)]
    \bE Z_{L-1}^2\leq
    Cd^{2L}.
\end{align*}
Dropping the nonnegative term in
\eqref{eq:quadratic-intermediate} now gives
\begin{align}\label{eq:quadratic-lower-bound}
    \bE U_{i,t,L}(o)
    \geq
    \bE\sum_{x\sim o}g_{t,L}(o,x)
    -
    C(\mu_0-t)d^{2L}
    -
    Cd^{2L}
    \left(\frac d{\mu_0^2}\right)^L.
\end{align}

It remains to show that the first term in
\eqref{eq:quadratic-lower-bound} outweighs the last two. Root the tree at a
uniformly chosen oriented core edge $(o,x)$. After deleting this edge,
the two resulting rooted trees are independent Galton--Watson trees
with offspring distribution \eqref{eq:forward-offspring}. Therefore,
for independent copies $W_o,W_x$ of the variable in
\Cref{lem:branching-limit},
\[
\begin{aligned}
    \frac{Z_{L-1}(o\mid x)}{d^{L-1}}
    &\longrightarrow W_o,
    &\qquad
    \frac{Z_{L-1}(x\mid o)}{d^{L-1}}
    &\longrightarrow W_x,\\
    \frac{Z_L(o)}{d^{L-1}}
    &\longrightarrow dW_o+W_x,
    &
    \frac{Z_L(x)}{d^{L-1}}
    &\longrightarrow W_o+dW_x
\end{aligned}
\]
in $L^2$. The last two limits follow by splitting each generation
according to which side of the deleted edge it lies on.

Write the signal eigenvalue relative to the Kesten--Stigum threshold
as
\[
    a\deq \frac{\mu_i^+}{\sqrt d}=\sqrt{d}\rho>1.
\]
Substituting the four preceding limits into
\eqref{eq:edge-contribution} at $t=\sqrt d$, and using the
independence of $W_o,W_x$ together with $\bE W=1$, gives
\begin{align}\label{eq:edge-limit}
    \lim_{L\to\infty}
    \frac{\bE g_{\sqrt d,L}(o,x)}{d^{2L-2}}
    =
    \frac{\sqrt d\,(a-1)}{a^2}
    \left[
        a(da-1)-d(a-1)\bE W^2
    \right].
\end{align}
By \eqref{eq:W-moments}, the expression in brackets is strictly
larger than
\[
    a(da-1)-2d(a-1)
    =
    d(a-1)^2+d-a.
\]
This is positive because
\[
    1<a\leq\sqrt d<d.
\]
Thus the limit in \eqref{eq:edge-limit} is strictly positive.

Returning from the oriented-edge law to the vertex-rooted law,
\[
\bE[\sum_{x\sim o}g_{t,L}(o,x)]=\bE[d_o]\bE[g_{t,L}(o,x)\mid o\sim x].
\]
The normalized expectation $\frac{1}{d^{2L-2}}\bE[\sum_{x\sim o}g_{t,L}(o,x)]$ is a linear combination of $1$, $t$, and
$t^{-1}$ whose coefficients converge as $L\to\infty$. The convergence
is therefore uniform for $t$ in a compact neighborhood of $\sqrt d$. Therefore, there are $\overline\epsilon>0$, $\kappa>0$, and $L_1$ such that
\[
    \bE\sum_{x\sim o}g_{t,L}(o,x)
    \geq
    \kappa d^{2L}
\]
whenever
$\sqrt d\leq t\leq\sqrt d+\overline\epsilon$ and $L\geq L_1$.

Going back to \eqref{eq:quadratic-lower-bound}, choose $\mu_0>\sqrt d$ sufficiently close to $\sqrt d$ that
\[
    C(\mu_0-\sqrt d)\leq\frac{\kappa}{4},
\]
and then choose
\[
    0<\epsilon_0
    <
    \min\{\overline\epsilon,\mu_0-\sqrt d\}.
\]
Finally, choose one fixed $L\geq L_1$ such that
\[
    C\left(\frac d{\mu_0^2}\right)^L
    \leq
    \frac{\kappa}{4}.
\]
It follows from \eqref{eq:quadratic-lower-bound} that, for
$\sqrt d\leq t\leq\sqrt d+\epsilon_0$,
\[
    \bE U_{i,t,L}(o)
    \geq
    \frac{\kappa}{2}d^{2L}.
\]

The set of informative eigenvalues is finite, so we may take the
minimum of the positive constants and the maximum of the required
values of $L$. This gives choices that work simultaneously for all
$i\in[r_+]$.
\end{proof}
The same orthogonality also removes all cross terms.

\begin{lemma}\label{lem:cross-terms}
For $i\neq j$, conditional on the unlabelled tree,
\[
    \bE[\psi_{i,L}(o)\psi_{j,L}(o)\mid T^{(2)}]=0
\]
and
\[
    \bE[
        \psi_{i,L}(o)[H(t)\psi_{j,L}](o)
        \mid T^{(2)}
    ]=0.
\]
\end{lemma}

\begin{proof}
After expanding the test vectors, every summand is a product of a
variable in index $i$ and one in index $j$. Its
conditional expectation is zero by
\Cref{lem:innovation-orthogonality}.
\end{proof}

\section{Passing from the limiting tree to the full graph}

Fix $\mu_0$, $L$, and $\epsilon_0$ as in
\Cref{prop:tree-negative}. On $ G^{(2)}$, define the test vectors by
\eqref{eq:test-vector}, with all distances and generation sizes taken
in $ G^{(2)}$. As above, set $\psi_{i,L}(o)=0$ whenever the
neighborhood required by \eqref{eq:test-vector} is not a tree.

We first show that, given the properties of our test vectors, we can find $r_+$ negative eigenvalues. Therefore, take $\psi_{i,L}^\Gamma$ to be the vector defined in \eqref{eq:test-vector} with reference to graph $\Gamma$.
\begin{proposition}\label{prop:finite-core}
With high probability, the vectors $\{\psi_{i,L}^{G^{(2)}}\}_{i\in[r_+]}$ span an $r_+$ dimensional subspace. Moreover, there is a constant $c>0$ such that, for every $t=t_n$ satisfying
    $\sqrt d\leq t\leq\sqrt d+\epsilon_0$,
 every vector $\chi\neq 0$ in the span of $\{\psi_{i,L}^{G^{(2)}}\}_{i\in[r_+]}$ satisfies
\[
\frac{\chi^\top H(t,G^{(2)})\chi}{\chi^\top \chi}\leq- c.
\]

\end{proposition}

\begin{proof}
 By \Cref{prop:local-averages} 
and
\Cref{lem:cross-terms}, for every $i,j\in[r_+]$,
\[
    \frac{\langle\psi_{i,L}^{G^{(2)}},\psi_{j,L}^{G^{(2)}}\rangle}
         {|V( G^{(2)})|}
    =
    \begin{cases}
        \bE\psi_{i,L}^{T^{(2)}}(o)^2+\oo_{n}(1),&i=j,\\
        \oo_n(1),&i\neq j.
    \end{cases}
\]
Similarly, by \eqref{eq:local-to-global-U},
\Cref{prop:tree-negative}, and
\Cref{lem:cross-terms}, for $U_{i,t,L}$ again defined for $T^{(2)}$,
\[
    \frac{
        \langle \psi_{i,L}^{G^{(2)}},
        H(t, G^{(2)})\psi_{j,L}^{G^{(2)}}
    \rangle}{|V( G^{(2)})|}
    =
    \begin{cases}
        -t\,\bE U_{i,t,L}(o)+\oo_n(1),&i=j,\\
        \oo_n(1),&i\neq j.
    \end{cases}
\]
These conclusions remain valid when $t=t_n$ varies in the stated
interval, since the dependence on $t$ involves only finitely many
local functions with uniformly bounded deterministic coefficients.

The limiting Gram matrix is positive definite, while by \Cref{prop:tree-negative}, the diagonal entries of the limiting quadratic-form
matrix are bounded above by
$-\sqrt d\,c$. Since $r_+$ is fixed, with probability
$1-\oo_n(1)$ the test vectors are linearly independent, and the Rayleigh quotient of
$H(t, G^{(2)})$ on their span is  bounded above by a negative
constant.

\end{proof}
We now give the harmonic extension of a vector from the 2-core to the entire graph. 
\begin{definition}\label{def:harmonic-extension}
Let $G$ be a finite graph, let $t\neq0$, and let
$f\in\RR^{V(G^{(2)})}$. Its $t$-harmonic extension
$\widetilde f\in\RR^{V(G)}$ is defined by setting
\[
    \widetilde f=f
    \quad\text{on }G^{(2)},
\]
setting $\widetilde f=0$ on every component with empty $2$-core, and
recursively setting
\[
    \widetilde f(x)=t^{-1}\widetilde f(p_u(x))
\]
along every tree attached to $G^{(2)}$, where $u$ denotes the root of the tree.
\end{definition}

\begin{lemma}\label{lem:harmonic-extension}
For every vector $f$ on $G^{(2)}$, its extension satisfies
\begin{align}\label{eq:harmonic-form-identity}
    \widetilde f^{\top}H(t,G)\widetilde f
    =f^{\top}H(t,G^{(2)})f.
\end{align}
\end{lemma}

\begin{proof}
Remove the vertices outside the $2$-core one leaf at a time. If $x$ is
a leaf with parent $y$, the change in the quadratic form when $x$ is
removed is
\[
    \widetilde f(y)^2+t^2\widetilde f(x)^2
    -2t\widetilde f(x)\widetilde f(y)
    =\bigl(\widetilde f(y)-t\widetilde f(x)\bigr)^2=0.
\]
Repeating this removes every
pendant tree and proves \eqref{eq:harmonic-form-identity}.
\end{proof}

We now will show that the norm does not drastically change when we pass to this extension. For $v\in V( G^{(2)})$, let $T_v$ be the rooted tree consisting of
$v$ and all vertices outside $ G^{(2)}$ whose unique path to $ G^{(2)}$ 
meets the core at $v$. Let $|T_v|$ denote the vertex number of $T_v$.
\begin{lemma}\label{lem:pendant-tree-second-moment}
There is a constant $C>0$ such that, with probability
$1-\oo_n(1)$,
\[
    \sum_{v\in V( G^{(2)})}|T_v|^2\leq Cn.
\]
\end{lemma}

\begin{proof}
By \Cref{lem:core-law}, the limiting pendant tree is a subcritical
Galton--Watson tree with offspring distribution $\Pois(dq)$, where
$dq<1$. The corresponding finite-graph exploration is uniformly
subcritical, and hence the sizes of the pendant trees have a uniform
exponential tail. In particular, for every fixed $k$,
\[
    \frac1{|V( G^{(2)})|}
    \sum_{v\in V( G^{(2)})}
        (|T_v|\wedge k)^2
    \overset{\bP}{\longrightarrow}
    \bE[(|T|\wedge k)^2],
\]
by \Cref{prop:local-averages}, while the exponential tail  $\bE[|T|^2]-\bE[(|T|\wedge k)^2]$ makes the
truncation error uniformly negligible as $k\to\infty$. Therefore
\[
    \frac1{|V( G^{(2)})|}
    \sum_{v\in V( G^{(2)})}|T_v|^2
    \overset{\bP}{\longrightarrow}
    \bE|T|^2<\infty.
\]
Since $|V( G^{(2)})|\leq n$, the result follows.
\end{proof}
We will now show that for the fixed-radius test vectors, the norm of the extension has a
direct limiting formula. From now on, we will write $\psi_{i,L}\deq \psi_{i,L}^{G^{(2)}}$ for notational convenience.

\begin{lemma}\label{lem:extension-norm}
Let $t=t_n$ be as in \Cref{prop:finite-core}, and let
$\widetilde\psi_{i,L}$ be the $t$-harmonic extension of
$\psi_{i,L}$ from $G^{(2)}$ to $G$. Then
\begin{align}\label{eq:extension-norm}
    \frac{\|\widetilde\psi_{i,L}\|_2^2}{|V( G^{(2)})|}
    -\frac{1}{1-dq/t^2}\cdot \frac{\|\psi_{i,L}\|^2}{|V( G^{(2)})|}
    \overset \bP\longrightarrow0.
\end{align}
In particular,  $\|\widetilde\psi_{i,L}\|_2^2$ is bounded above and below by positive
constant multiples of $|V( G^{(2)})|$, uniformly in $n$.
\end{lemma}

\begin{proof}
Our extension is
\begin{align}\label{eq:extension-norm-sum}
    \|\widetilde\psi_{i,L}\|_2^2
    =
    \sum_{v\in V( G^{(2)})}
    \psi_{i,L}(v)^2
    \sum_{x\in T_v}t^{-2\dist(x,v)}.
\end{align}

Fix $K\geq0$ and first retain only the vertices $x\in T_v$ satisfying
$\dist(x,v)\leq K$. The resulting summand depends only on a
fixed-radius marked neighborhood of $v$. Therefore, 
\Cref{prop:local-averages} gives
\begin{align}\label{eq:truncated-extension-norm}
\begin{split}
    &\frac{1}{|V( G^{(2)})|}
    \sum_{v\in V( G^{(2)})}
    \psi_{i,L}(v)^2
    \sum_{\substack{x\in T_v\\ \dist(x,v)\leq K}}
        t^{-2\dist(x,v)}\\
    &\qquad
    -
    \bE\psi_{i,L}(o)^2
    \sum_{k=0}^K
        \left(\frac{dq}{t^2}\right)^k
    \overset \bP\longrightarrow0.
\end{split}
\end{align}
To see this,  by \Cref{lem:core-law}, in the limiting marked neighborhood
the pendant tree is independent of the labelled $2$-core and has
offspring distribution $\Pois(dq)$. Its expected number of vertices
at depth $k$ is therefore $(dq)^k$.

It remains to show that truncating at depth $K$ introduces a uniformly
small error. Since $t^2\geq d$, for every $v\in V( G^{(2)})$,
\[
    \sum_{\substack{x\in T_v\\ \dist(x,v)>K}}
        t^{-2\dist(x,v)}
    \leq
    d^{-K}|T_v|.
\]
Moreover,
\[
    \frac{|V( G^{(2)})|}{n}
    \overset{\bP}{\longrightarrow}
    \bP(\Pois(\eta)\geq2)>0.
\]

Consequently, by Cauchy--Schwarz,
\begin{align}\label{eq:extension-tail-bound}
\begin{split}
    &\frac{1}{|V( G^{(2)})|}
    \sum_{v\in V( G^{(2)})}
        \psi_{i,L}(v)^2
        \sum_{\substack{x\in T_v\\ \dist(x,v)>K}}
            t^{-2\dist(x,v)}\\
    &\qquad
    \leq
    d^{-K}
    \left(
        \frac{1}{|V( G^{(2)})|}
        \sum_{v\in V( G^{(2)})}\psi_{i,L}(v)^4
    \right)^{1/2} 
    \left(
        \frac{1}{|V( G^{(2)})|}
        \sum_{v\in V( G^{(2)})}|T_v|^2
    \right)^{1/2}.
\end{split}
\end{align}
Because $L$ is fixed, another application of
\Cref{prop:local-averages} gives
\[
    \frac{1}{|V( G^{(2)})|}
    \sum_{v\in V( G^{(2)})}\psi_{i,L}(v)^4
    \overset \bP\longrightarrow
    \bE\psi_{i,L}(o)^4<\infty.
\]
Together with \Cref{lem:pendant-tree-second-moment}, the two factors
on the right-hand side of \eqref{eq:extension-tail-bound} are
$O_n(1)$. Therefore, for every $\delta>0$,
\begin{align}\label{eq:extension-tail-vanishes}
    \lim_{K\to\infty}\limsup_{n\to\infty}
    \bP\left(
        \frac{1}{|V( G^{(2)})|}
        \sum_{v\in V( G^{(2)})}
        \psi_{i,L}(v)^2
        \sum_{\substack{x\in T_v\\ \dist(x,v)>K}}
            t^{-2\dist(x,v)}
        >\delta
    \right)
    =0.
\end{align}

The corresponding tail of the limiting geometric series is also
uniformly small. Indeed, since $dq/t^2\leq q<1$,
\[
    0
    \leq
    \frac{\bE\psi_{i,L}(o)^2}{1-dq/t^2}
    -
    \bE\psi_{i,L}(o)^2
    \sum_{k=0}^K
        \left(\frac{dq}{t^2}\right)^k
    \leq
    \frac{\bE\psi_{i,L}(o)^2}{1-q}q^{K+1}.
\]
Combining this bound with
\eqref{eq:truncated-extension-norm} and
\eqref{eq:extension-tail-vanishes}, and then letting $K\to\infty$,
proves \eqref{eq:extension-norm}.

Finally,
\[
    1
    \leq
    \frac{1}{1-dq/t^2}
    \leq
    \frac{1}{1-q},
\]
uniformly for $t^2\geq d$. Since
$0<\bE\psi_{i,L}(o)^2<\infty$, the final assertion follows.
\end{proof}

We are now ready to prove the main theorem.
\begin{proof}[Proof of \Cref{thm:main}]
The upper bounds were proved in \Cref{cor:thm-main-ub}. For the lower bounds, fix any $0\leq  \epsilon\leq  \epsilon_0$, and let $t=\sqrt{d}+\epsilon$.
With probability $1-\oo_n(1)$, 
\Cref{prop:finite-core} gives an
$r_+$-dimensional subspace, spanned by $(\psi_{i,L})_{i\in[r_+]}$, on which
the Rayleigh quotient of $H(t, G^{(2)})$ is at most a negative constant. We now extend this basis to $(\widetilde \psi_{i,L})_{i\in[r_+]}$ on $G$ using
\Cref{def:harmonic-extension}.  By \Cref{lem:extension-norm} and the fact that $r_+$ is
fixed, for all $\widetilde \psi\deq \sum_{i=1}^{r_+}a_i\widetilde\psi_{i,L}$, we have
\[
   \left\|\widetilde \psi\right\|_2^2= \left\|\sum_{i=1}^{r_+}a_i\widetilde\psi_{i,L}\right\|_2^2
    \leq
    C|V( G^{(2)})|\sum_{i=1}^{r_+}a_i^2.
\]
Note that $\widetilde \psi$ is the harmonic extension of $\psi\deq \sum_{i=1}^{r_+}a_i \psi_{i,L}$.
By
\Cref{lem:harmonic-extension} and Proposition \ref{prop:finite-core}, we have

\be\label{eq:lasteq}
  \widetilde \psi^{\top} H(t,G) \widetilde \psi= \left(\sum_{i=1}^{r_+}a_i \psi_{i,L}\right)^{\top} H(t,G^{(2)}) \left(\sum_{i=1}^{r_+}a_i \psi_{i,L}\right)\leq   -c|V( G^{(2)})|\sum_{i=1}^{r_+}a_i^2.
\ee
The harmonic extension is injective, so the span of $(\widetilde \psi_{i,L})_{i\in[r_+]}$ still has dimension $r_+$.
Thus $H(t, G)$ has at least $r_+$ eigenvalues below a fixed negative
constant $-c_0$. This proves
\[
    N_{c_0}(\sqrt d+\epsilon)\geq r_+
\]
for all $\sqrt{d}\leq t\leq \sqrt{d}+\epsilon_0$. Up to replacing $\epsilon_0$ by $\min\{c_0,\epsilon_0\}$, this gives
\[
N_\epsilon(\sqrt{d}),\, N(\sqrt{d}+\epsilon)\geq r_+\qquad \text{for all $0\leq\epsilon\leq\epsilon_0$.}
\]
Together with \Cref{cor:thm-main-ub}, this finishes the proof.

\end{proof}
\begin{remark}
    
For a negative informative eigenvalue, make the following changes. Take $t<0$, $\mu_0<0$, and
$\rho_i<0$. In each rooted tree calculation, replace
$X_{u,i}$ by $(-1)^{\dist(o,u)}X_{u,i}$. If
$s=-t$, $\nu_0=-\mu_0$, and $\zeta_i=-\rho_i$, then the test vector becomes
exactly \eqref{eq:test-vector} with the positive parameters
$s,\nu_0,\zeta_i$. For a neighbor of the root the transformed test vector
changes sign, so the local quadratic contribution for $H(t)$ from  \eqref{eq:lasteq} is the
same as the positive-parameter contribution for $H(s)$. The harmonic
extension identity is unchanged, and its norm depends only on $t^2$.
Thus the same argument gives
\[
   N_{\epsilon}(-\sqrt d),\, N(-\sqrt d-\epsilon)\leq r_-.
\]
\end{remark}

\bibliographystyle{amsplain}
\bibliography{ref}

\end{document}